\documentclass[11pt,reqno]{amsart}

\usepackage[T1]{fontenc}
\usepackage{lmodern}
\usepackage{amsmath,amssymb,mathtools}
\usepackage{microtype}
\usepackage{needspace}
\usepackage[numbers,sort&compress]{natbib}

\usepackage[hidelinks]{hyperref}
\hypersetup{
  pdftitle={Scott topologies on products of countable complete Heyting algebras}
  pdfauthor={Xiaoquan Xu},
  pdfsubject={Scott topologies, products, sobriety, meet continuity, the dichotomy theorem and cardinal spectra},
  pdfborder={0 0 0}
}

\newtheorem{theorem}{Theorem}[section]
\newtheorem{proposition}[theorem]{Proposition}
\newtheorem{lemma}[theorem]{Lemma}
\newtheorem{corollary}[theorem]{Corollary}
\theoremstyle{definition}
\newtheorem{definition}[theorem]{Definition}
\newtheorem{question}[theorem]{Question}
\newtheorem{example}[theorem]{Example}
\newtheorem{remark}[theorem]{Remark}
\numberwithin{equation}{section}

\newcommand{\OO}{\mathcal O}
\newcommand{\pt}{\operatorname{pt}}
\newcommand{\N}{\mathbb N}
\newcommand{\Specsp}{\operatorname{NSpec}_{\mathrm{sp}}}
\newcommand{\Specfr}{\operatorname{NSpec}_{\mathrm{fr}}}
\newcommand{\ssp}{\mathfrak{s}_{\mathrm{sp}}}
\newcommand{\sfr}{\mathfrak{s}_{\mathrm{fr}}}
\newcommand{\cc}{\mathfrak{c}}
\title[Scott topologies of complete Heyting algebras]
{Scott topologies on products of countable complete Heyting algebras}
\author{Xiaoquan Xu}
\address{School of Computer Information Engineering,
Nanchang Institute of Technology, Nanchang 330044, China}
\email{xiqxu2002@163.com}
\urladdr{https://orcid.org/0000-0003-1159-8477}

\begin{document}

\begin{abstract}
We prove that the Scott topology commutes with arbitrary products of
countably presented frames and that every such product has a sober Scott
space. In particular, this holds for arbitrary products of countable
complete Heyting algebras. For a family of consonant spaces, Scott-product
compatibility of their open-set lattices is equivalent to consonance of
their topological sum. Countable generation does not suffice: a countably
generated spatial frame can have a non-sober Scott space and fail the
product identity for its square. We also show that the cardinal spectra
of Scott non-sober frames and spatial frames are upward closed and, under
the Continuum Hypothesis, consist of all uncountable cardinals. Finally,
every Artinian $T_0$ web space is a $B$-space; if it is also a $d$-space,
it carries the Scott topology of an algebraic dcpo. Consequently, every
Artinian meet-continuous dcpo is algebraic. This yields a dichotomy theorem:
a dcpo with a non-sober Scott space must fail meet continuity or Artinianity.
\end{abstract}

\maketitle

\section{Introduction}\label{sec:intro}

For posets $P_i$ indexed by a set $I$, the product of the Scott topologies is
always contained in the Scott topology of the product order:
\begin{equation}\label{eq:general-inclusion}
  \prod_{i\in I}\sigma(P_i)
  \subseteq
  \sigma\!\left(\prod_{i\in I}P_i\right).
\end{equation}
The inclusion can be strict, already for two complete lattices. Determining
when equality holds is a basic problem in domain theory and non-Hausdorff
topology, and is closely related to joint Scott continuity of lattice
operations and to sobriety; see, for example,
\citep{GierzEtAl2003,LawsonXu2024}. Throughout this paper, the Scott space
of a product order is distinguished from the ordinary topological product
of the Scott spaces.

A natural class on which to study this problem is the class of complete
Heyting algebras. These are precisely frames, that is, complete lattices
in which finite meets distribute over arbitrary joins. Countability and
the frame law play different roles in the sobriety problem. Miao, Xi, Li
and Zhao constructed countable complete lattices, including a distributive
one, with non-sober Scott spaces \citep{MiaoXiLiZhao2023}. Such examples
cannot be frames: every countable frame has a sober Scott space, as also
follows from the results below. On the other hand, the frame law alone
does not guarantee Scott sobriety, as shown by Xu, Xi and Zhao
\citep{XuXiZhao2021}. Even within the countable class the product problem
is not reduced to the theory of continuous lattices: Jia and Xi
\citep{JiaXi2026} constructed countable frames that are not continuous.

Our positive results concern \emph{countably presented frames}, a class
which contains all countable frames but also contains uncountable ones.
By work of Heckmann and de Brecht, these are exactly the frames isomorphic
to $\OO(X)$ for quasi-Polish spaces $X$; see
\citep{Heckmann2015,deBrecht2013,deBrechtKawai2019}. This representation
converts the product problem into a topological question about consonance.
For a topological sum $X=\coprod_{i\in I}X_i$, the canonical frame
isomorphism
\[
  \OO(X)\cong\prod_{i\in I}\OO(X_i)
\]
identifies the Scott topology on $\OO(X)$ with that on the product order.
We prove that consonance of $X$ makes this topology equal to the product
of the individual Scott topologies. Conversely, when each $X_i$ is
consonant, the Scott-product identity implies that $X$ is consonant.
Thus, for consonant summands, compatibility has an exact characterization
in terms of their topological sum.

Quasi-Polish spaces are LCS-complete, every LCS-complete space is consonant,
and arbitrary topological sums of LCS-complete spaces are LCS-complete
\citep{deBrechtEtAl2019}. We consequently obtain the following theorem.

\begin{theorem}\label{thm:main-intro}
Let $(L_i)_{i\in I}$ be an arbitrary family of countably presented frames.
Then
\[
  \Sigma\!\left(\prod_{i\in I}L_i\right)
  =\prod_{i\in I}\Sigma L_i.
\]
Here the left-hand side carries the Scott topology of the coordinatewise
order and the right-hand side the ordinary topological product.
\end{theorem}

In particular, Scott topology commutes with arbitrary products of
countable complete Heyting algebras. No restriction on the index set is
needed. We also prove that a product of countably presented frames is
itself countably presented if and only if only countably many factors are
nontrivial. Hence the Scott-product identity persists for products that
leave the countably presented class. The identity also implies that the
Scott space of every such product is sober. For a single countably
presented frame this sobriety conclusion was already proved by de Brecht
and Kawai \citep[Corollary~8.6]{deBrechtKawai2019}; here it follows within
a product theorem that also applies when the product is not countably
presented.

The distinction between \emph{countable presentation} and \emph{countable
generation} is essential. He and Zhao \citep[Proposition~4.11]{HeZhao2026}
proved that the lower topology of a particular countable dcpo, regarded
as a frame, has a non-sober Scott space. This topology has a countable
base, so its frame of open sets is countably generated and spatial.
Combining this example with Shelah's theorem on the cardinality of a
countably based topology \citep{Shelah1993}, we obtain a countably
generated spatial frame $L_0$ such that
\[
  |L_0|=\cc,~ \Sigma L_0\text{ is non-sober},
\]
and
\[
  \sigma(L_0)\times\sigma(L_0)
  \subsetneq\sigma(L_0\times L_0).
\]
This is a consequence of their construction, rather than a new independent
counterexample. It shows that countable generation cannot replace
countable presentation in Theorem~\ref{thm:main-intro}, even for spatial
frames and two equal factors.

The same example leads to a cardinal-spectrum problem. Let $\sfr$ and
$\ssp$ denote the least cardinalities of a frame and of a spatial frame,
respectively, whose Scott spaces are non-sober. We show that both spectra
are upward closed and that
\begin{equation}\label{eq:intro-bounds}
  \aleph_1\leq\sfr\leq\ssp\leq\cc.
\end{equation}
The lower bound comes from the countable-frame result, while the upper
bound follows from the He--Zhao example. Under the Continuum Hypothesis,
both least cardinalities are $\aleph_1$, and spatial examples exist in
every uncountable cardinality. Moreover, when $\cc>\aleph_1$, a spatial
frame of cardinality $\aleph_1$ with a non-sober Scott space cannot admit
a representation by a countably based space. Thus the cardinal question
also clarifies a limitation of countably based constructions.

A complementary positive result arises from a descending chain condition
rather than a cardinality restriction. The connection with complete Heyting
algebras is provided by the closed-set lattice: a dcpo is meet-continuous
if and only if its lattice of Scott-closed sets is a complete Heyting
algebra, a characterization due to Kou, Liu and Luo
\citep{KouLiuLuo2003}. We recall this characterization, its extension to
posets, and its relation to Ern\'e's web spaces in
Theorem~\ref{thm:mc-characterization}.

We prove that a $T_0$ web space with Artinian specialization order has a
base of open principal upsets, and is therefore a $B$-space. More
generally than the Scott-space application alone, if a $d$-space $X$ has
Artinian specialization order and its closed-set lattice $\Gamma(X)$ is a
complete Heyting algebra, then its specialization order $\Omega X$ is an
algebraic dcpo and its original topology is exactly $\sigma(\Omega X)$
(Theorem~\ref{thm:artinian-dspace}). In particular, every Artinian
meet-continuous dcpo is algebraic and hence Scott sober. The resulting
\emph{Dichotomy theorem for Scott non-sober dcpos}
(Theorem~\ref{thm:dichotomy}) states that a dcpo with a non-sober Scott
space must fail meet continuity or Artinianity, possibly both. We use this
theorem to compare seven familiar non-sober constructions,
including explicit ordinal-rank and finite-intersection arguments for the
Artinianity of the extended Johnstone dcpo and Isbell's complete lattice. Together with
the countable-frame theorem, it motivates asking whether countability and
meet continuity suffice for Scott sobriety, first for complete lattices
and then for arbitrary dcpos.

The paper is organized as follows. Section~\ref{sec:prelim} recalls the
required frame-theoretic and topological facts, including the standard
sobriety criterion associated with the Scott-product identity.
Section~\ref{sec:consonance} proves the consonance characterization.
Section~\ref{sec:products} establishes the product theorem, determines
when the product remains countably presented, and derives sobriety.
Section~\ref{sec:cardinal} treats countable generation and the cardinal
spectra. Section~\ref{sec:artinian} recalls the meet-continuity
characterization, proves the Artinian web-space and $d$-space theorems,
and applies the dichotomy theorem to the non-sober examples.
Section~\ref{sec:questions} records the remaining questions for countable
meet-continuous complete lattices and dcpos, and for frames of cardinality
$\aleph_1$.

\section{Preliminaries}\label{sec:prelim}

We use standard terminology from domain theory and frame theory; see
\citep{GierzEtAl2003,GoubaultLarrecq2013}.

Let $P$ be a poset.  A nonempty subset $D\subseteq P$ is \emph{directed} if
every finite subset of $D$ has an upper bound in $D$.  A subset $U\subseteq P$
is \emph{Scott open} if it is an upper set and, whenever a directed set $D$
has a supremum with $\bigvee D\in U$, one has $D\cap U\ne\varnothing$.  The
Scott topology is denoted by $\sigma(P)$ and the corresponding Scott space by
$\Sigma P$. A poset is a \emph{dcpo} if every directed subset has a
supremum. For dcpos, a map is Scott-continuous if and only if it is
monotone and preserves directed suprema.

For a topological space $X$, write $\OO(X)$ for its frame of open subsets.  A
subset of a topological space is \emph{saturated} if it is an intersection of open
sets. A nonempty subset $A$ is \emph{irreducible} if any two open sets
meeting $A$ have intersection meeting $A$. A $T_0$-space is
\emph{sober} if every irreducible closed subset is the closure of a
unique point. We use the standard fact that sobriety is preserved by
continuous retracts \citep{GoubaultLarrecq2013}.

A complete lattice $L$ is \emph{meet-continuous} if
\[
  a\wedge\bigvee D=\bigvee_{d\in D}(a\wedge d)
\]
for every $a\in L$ and every directed subset $D\subseteq L$.  A \emph{frame}, equivalently a \emph{complete Heyting algebra}, is a complete
lattice $L$ satisfying the stronger identity
\[
  a\wedge\bigvee S=\bigvee_{s\in S}(a\wedge s)
\]
for every $a\in L$ and $S\subseteq L$.  Thus every frame is
meet-continuous.  A point of $L$ is a frame homomorphism
$p:L\to\{0,1\}$.  The space of points $\pt(L)$ has basic open sets
\[
  \widehat a=\{p\in\pt(L):p(a)=1\},~ a\in L.
\]
A frame is \emph{spatial} if it is isomorphic to the frame of open sets
of a topological space. If $L$ is spatial, the canonical map $a\mapsto\widehat a$ is a frame
isomorphism from $L$ onto $\OO(\pt(L))$, and $\pt(L)$ is sober; see
\citep[Chapter~V]{GierzEtAl2003}.

For a compact saturated subset $K$ of a space $X$, put
\[
  \Phi_X(K)=\{U\in\OO(X):K\subseteq U\}.
\]
The family $\Phi_X(K)$ is Scott open in $\OO(X)$.

\begin{definition}\label{def:consonant}
A space $X$ is \emph{consonant} if, for every Scott-open
$\mathcal H\subseteq\OO(X)$ and every $U\in\mathcal H$, there exists a
compact saturated set $K\subseteq U$ such that
\[
  U\in\Phi_X(K)\subseteq\mathcal H.
\]
The empty compact set is allowed.
\end{definition}

A space is \emph{LCS-complete} if it is homeomorphic to a $G_\delta$-subspace
of a locally compact sober space \citep{deBrechtEtAl2019}.  A quasi-Polish
space is a countably based completely quasi-metrizable space
\citep{deBrecht2013}; equivalently, it is a countably based LCS-complete space
\citep{deBrechtEtAl2019}.  We use the following facts.

\begin{theorem}[de Brecht et al.]
\label{thm:lcs}
Every LCS-complete space is sober and consonant.  Moreover, the topological
sum of an arbitrary family of LCS-complete spaces is LCS-complete.
\end{theorem}

\begin{proof}[Reference]
These are Propositions~7.1, 12.1 and~15.1 of
\citep{deBrechtEtAl2019}.
\end{proof}

\begin{definition}\label{def:countably-presented}
A frame $L$ is \emph{countably presented} if it admits a presentation in
the category of frames with countably many generators and countably many
relations. A frame is \emph{countably generated} if some countable subset
generates it under finite meets and arbitrary joins. Thus countable
presentability implies countable generation; the converse need not hold.
See
\citep{deBrechtKawai2019,Heckmann2015}.
\end{definition}

The following representation theorem combines results of Heckmann and de Brecht.

\begin{theorem}\label{thm:cp-quasipolish}
For a frame $L$, the following are equivalent:
\begin{enumerate}
\item $L$ is countably presented;
\item $L\cong\OO(X)$ for some quasi-Polish space $X$.
\end{enumerate}
In particular, every countably presented frame is spatial and its point space
is quasi-Polish.
\end{theorem}

\begin{proof}[Reference]
Heckmann proved spatiality of countably presented locales and represented
their point spaces by countably based $\Pi^0_2$ spaces
\citep{Heckmann2015}.  Combined with de Brecht's characterization of
quasi-Polish spaces \citep{deBrecht2013}, this yields the equivalence above;
see also the explicit formulation in \citep{deBrechtKawai2019}.
\end{proof}

We also recall the following standard sobriety criterion; see
\citep[Corollary~II-1.12]{GierzEtAl2003}. We include a proof because its
contrapositive links the two parts of the paper.

\begin{proposition}\label{prop:joint-sup-sober}
Let $L$ be a complete lattice.  If
\[
  \sigma(L\times L)=\sigma(L)\times\sigma(L),
\]
then $\Sigma L$ is sober.
\end{proposition}

\begin{proof}
The binary join map
\[
  \vee:L\times L\longrightarrow L,~ (x,y)\longmapsto x\vee y,
\]
preserves directed suprema and is therefore Scott-continuous from
$\Sigma(L\times L)$ to $\Sigma L$. Under the assumed identity it is
continuous from the ordinary product $\Sigma L\times\Sigma L$ to
$\Sigma L$.

Let $A$ be a nonempty irreducible Scott-closed subset of $L$. We show that
$A$ is closed under binary joins. Take $a,b\in A$. If $a\vee b\notin A$,
then $U=L\setminus A$ is a Scott-open neighborhood of $a\vee b$.
Continuity of $\vee$ gives Scott-open neighborhoods $V$ of $a$ and $W$ of
$b$ with
\[
  V\times W\subseteq\vee^{-1}(U).
\]
Since $A$ is irreducible and meets $V$ and $W$, choose
$c\in A\cap V\cap W$. Then $c=c\vee c\in U$, a contradiction. Hence
$a\vee b\in A$, and in particular $A$ is directed.

Scott closedness implies that $s=\bigvee A$ belongs to $A$. Since $A$ is
a lower set, $A=\mathord{{\downarrow}}s$. Thus $A$ is the closure of $s$,
and the generic point is unique because Scott spaces are $T_0$.
\end{proof}

\begin{corollary}\label{cor:product-strict}
If $L$ is a complete lattice with a non-sober Scott space, then
\[
  \sigma(L)\times\sigma(L)\subsetneq\sigma(L\times L).
\]
\end{corollary}

\begin{proof}
Combine \eqref{eq:general-inclusion} with the contrapositive of
Proposition~\ref{prop:joint-sup-sober}.
\end{proof}

\section{Consonance and Scott-product compatibility}\label{sec:consonance}

Let $(X_i)_{i\in I}$ be a family of spaces and put
$X=\coprod_{i\in I}X_i$.  The canonical map
\begin{equation}\label{eq:theta}
  \Theta:\OO(X)\longrightarrow\prod_{i\in I}\OO(X_i),
  ~ \Theta(U)=(U\cap X_i)_{i\in I},
\end{equation}
is a frame isomorphism.

\begin{theorem}\label{thm:consonance-characterization}
Let $X_i$ $(i\in I)$ be topological spaces, and set
$X=\coprod_{i\in I}X_i$.
\begin{enumerate}
\item If $X$ is consonant, then $\Theta$ is a homeomorphism
\[
   \Sigma\OO(X)\longrightarrow\prod_{i\in I}\Sigma\OO(X_i).
\]
Equivalently,
\begin{equation}\label{eq:scott-product-open}
  \sigma\!\left(\prod_{i\in I}\OO(X_i)\right)
  =\prod_{i\in I}\sigma(\OO(X_i)).
\end{equation}
\item If every $X_i$ is consonant and \eqref{eq:scott-product-open} holds,
then $X$ is consonant.
\end{enumerate}
Consequently, for a family of consonant spaces, the Scott-product identity
\eqref{eq:scott-product-open} holds if and only if their topological sum is
consonant.
\end{theorem}

\begin{proof}
Because $\Theta$ is an order isomorphism, it identifies the Scott topology on
$\OO(X)$ with the Scott topology on the product poset
$\prod_i\OO(X_i)$.  The product topology
$\prod_i\sigma(\OO(X_i))$ is always contained in the latter Scott topology,
since every coordinate projection preserves directed suprema.

Assume first that $X$ is consonant.  Let $\mathcal H$ be Scott open in
$\OO(X)$ and $U\in\mathcal H$.  Choose a compact saturated set $K\subseteq U$
such that
\[
  U\in\Phi_X(K)\subseteq\mathcal H.
\]
Since the summands form an open cover of $X$, compactness of $K$ implies
that
\[
  F=\{i\in I:K\cap X_i\ne\varnothing\}
\]
is finite.  Put $K_i=K\cap X_i$.  Since $X_i$ is clopen in $X$, each $K_i$
is compact and saturated in $X_i$.  Then
\begin{equation}\label{eq:compact-cylinder}
  \Theta(\Phi_X(K))
  =\prod_{i\in I}\Phi_{X_i}(K_i).
\end{equation}
For $i\notin F$, $K_i=\varnothing$ and
$\Phi_{X_i}(K_i)=\OO(X_i)$; for $i\in F$, $\Phi_{X_i}(K_i)$ is Scott
open.  Hence the right-hand side of
\eqref{eq:compact-cylinder} is a basic product-open neighborhood of
$\Theta(U)$ contained in $\Theta(\mathcal H)$.  Thus $\Theta(\mathcal H)$ is
product open, proving (1).

Conversely, assume that every $X_i$ is consonant and that
\eqref{eq:scott-product-open} holds.  Let $\mathcal H$ be Scott open in
$\OO(X)$ and $U\in\mathcal H$.  Since $\Theta(\mathcal H)$ is product open,
there are a finite set $F\subseteq I$ and Scott-open families
$\mathcal H_i\subseteq\OO(X_i)$ $(i\in F)$ such that, writing
$U_i=U\cap X_i$,
\begin{equation}\label{eq:basic-neighborhood}
  \Theta(U)\in
  \left(\prod_{i\in F}\mathcal H_i\right)
  \times\left(\prod_{i\notin F}\OO(X_i)\right)
  \subseteq \Theta(\mathcal H).
\end{equation}
By consonance of $X_i$, choose compact saturated $K_i\subseteq U_i$ such
that
\[
  U_i\in\Phi_{X_i}(K_i)\subseteq\mathcal H_i
  ~ (i\in F).
\]
Let $K=\bigcup_{i\in F}K_i$, regarded as a subset of the topological sum
$X$.  Since $F$ is finite and the $X_i$ are clopen summands, $K$ is compact
and saturated in $X$; moreover, $K\subseteq U$.  If
$V\in\Phi_X(K)$, then $V\cap X_i\in\Phi_{X_i}(K_i)\subseteq\mathcal H_i$
for every $i\in F$.  Equation~\eqref{eq:basic-neighborhood} therefore gives
$V\in\mathcal H$.  Hence
\[
  U\in\Phi_X(K)\subseteq\mathcal H,
\]
and $X$ is consonant.
\end{proof}

The characterization immediately yields a large positive class.

\begin{corollary}\label{cor:lcs-products}
Let $(X_i)_{i\in I}$ be an arbitrary family of LCS-complete spaces.  Then
\[
  \sigma\!\left(\prod_{i\in I}\OO(X_i)\right)
  =\prod_{i\in I}\sigma(\OO(X_i)).
\]
\end{corollary}

\begin{proof}
By Theorem~\ref{thm:lcs}, the topological sum $\coprod_iX_i$ is
LCS-complete and hence consonant.  Apply
Theorem~\ref{thm:consonance-characterization}(1).
\end{proof}

The converse part also shows that the hypothesis cannot simply be weakened to
consonance of each summand.

\begin{corollary}\label{cor:strict-consonant}
There exist consonant spaces $X$ and $Y$ such that
\[
  \sigma(\OO(X)\times\OO(Y))
  \supsetneq
  \sigma(\OO(X))\times\sigma(\OO(Y)).
\]
\end{corollary}

\begin{proof}
As recalled in \citep[Section~13]{deBrechtEtAl2019}, there are consonant
spaces $X$ and $Y$ whose topological sum $X\coprod Y$ is not consonant.
That source attributes the example to Nogura and Shakhmatov.  If equality held above, Theorem~\ref{thm:consonance-characterization}(2)
would imply that $X\coprod Y$ is consonant, a contradiction.
\end{proof}

\section{Products of countably presented frames}\label{sec:products}

We now pass from spaces to frames.

\begin{theorem}\label{thm:main}
Let $(L_i)_{i\in I}$ be an arbitrary family of countably presented frames.
Then
\begin{equation}\label{eq:main}
  \Sigma\!\left(\prod_{i\in I}L_i\right)
  =\prod_{i\in I}\Sigma L_i.
\end{equation}
\end{theorem}

\begin{proof}
For each $i\in I$, Theorem~\ref{thm:cp-quasipolish} gives a quasi-Polish
space $X_i$ such that
\[
  L_i\cong\OO(X_i).
\]
Every quasi-Polish space is LCS-complete.  Hence
Corollary~\ref{cor:lcs-products} gives
\[
  \sigma\!\left(\prod_{i\in I}\OO(X_i)\right)
  =\prod_{i\in I}\sigma(\OO(X_i)).
\]
Transporting the topologies through the frame isomorphisms yields
\eqref{eq:main}.
\end{proof}

\begin{corollary}\label{cor:self-powers}
If $L$ is a countably presented frame and $\kappa$ is any cardinal, then
\[
  \Sigma(L^\kappa)=(\Sigma L)^\kappa.
\]
\end{corollary}

The original countable-frame result is now an immediate special case.

\begin{proposition}\label{prop:countable-cp}
Every countable frame is countably presented.
\end{proposition}

\begin{proof}
Let $L$ be a countable frame.  By Ern\'e's theorem
\citep[Lemma~1.1(1)]{Erne2019}, $L$ is spatial.  Thus $L\cong\OO(\pt(L))$, where
$\pt(L)$ is sober.  Since $L$ is countable, the topology
$\OO(\pt(L))$ is countable.  De Brecht proved that every sober space with a
countable topology is quasi-Polish \citep[Corollary~6.6]{deBrecht2018}.  Hence $\pt(L)$ is
quasi-Polish, and Theorem~\ref{thm:cp-quasipolish} implies that $L$ is
countably presented.
\end{proof}

\begin{corollary}\label{cor:countable-old}
For every family $(L_i)_{i\in I}$ of countable frames,
\[
  \Sigma\!\left(\prod_{i\in I}L_i\right)
  =\prod_{i\in I}\Sigma L_i.
\]
In particular, this holds for arbitrary powers of every countable frame.
\end{corollary}

\begin{remark}\label{rem:noncontinuous}
Corollary~\ref{cor:countable-old} applies to arbitrary powers of the
non-continuous countable frames constructed by Jia and Xi
\citep{JiaXi2026}.  Thus Scott-product compatibility, even for every
self-power, does not force continuity within the class of countable frames.
\end{remark}

The next result explains why the arbitrary-index version of
Theorem~\ref{thm:main} goes beyond the class appearing in its hypotheses.
Call a frame \emph{trivial} if it has one element.

\begin{theorem}\label{thm:product-cp}
Let $(L_i)_{i\in I}$ be a family of countably presented frames, and put
\[
  J=\{i\in I:L_i\text{ is nontrivial}\}.
\]
Then $\prod_{i\in I}L_i$ is countably presented if and only if $J$ is
countable.
\end{theorem}

\begin{proof}
Trivial factors do not affect the product, so it is enough to work with the
indices in $J$.  For each $i\in J$, choose a quasi-Polish space $X_i$ with
$L_i\cong\OO(X_i)$.  Since $L_i$ is nontrivial and spatial, $X_i$ is
nonempty.  Put
\[
  X=\coprod_{i\in J}X_i.
\]
Then
\begin{equation}\label{eq:product-open-sum}
  \prod_{i\in I}L_i\cong\OO(X).
\end{equation}

Suppose first that $J$ is countable.  Each $X_i$ is countably based, so $X$
is countably based.  By Theorem~\ref{thm:lcs}, $X$ is LCS-complete.  Hence
$X$ is quasi-Polish, and Theorem~\ref{thm:cp-quasipolish} together with
\eqref{eq:product-open-sum} shows that $\prod_iL_i$ is countably presented.

Conversely, suppose that $\prod_iL_i$ is countably presented.  The space $X$
is LCS-complete by Theorem~\ref{thm:lcs}, hence sober.  Therefore the
canonical map from $X$ to $\pt(\OO(X))$ is a homeomorphism.  By
\eqref{eq:product-open-sum} and the assumption, Theorem~\ref{thm:cp-quasipolish}
shows that $\pt(\OO(X))$, and therefore $X$, is quasi-Polish; in particular,
$X$ is countably based.

A topological sum of uncountably many nonempty spaces cannot be countably
based.  Indeed, if $\mathcal B$ were a countable base and $x_i\in X_i$ were
chosen for each $i\in J$, then for each $i$ there would be
$B_i\in\mathcal B$ with $x_i\in B_i\subseteq X_i$.  The sets $B_i$ are
nonempty and pairwise distinct.  Thus $J$ must be countable.
\end{proof}

\begin{corollary}\label{cor:uncountable-power-notcp}
Let $L$ be a nontrivial countably presented frame and let $\kappa$ be an
uncountable cardinal.  Then $L^\kappa$ is not countably presented, but
\[
  \Sigma(L^\kappa)=(\Sigma L)^\kappa.
\]
\end{corollary}

Thus countable presentability is not preserved by arbitrary products, while
the Scott-product identity is.  The same identity also yields sobriety well
beyond the countably presented case.

\begin{theorem}\label{thm:product-sober}
Let $(L_i)_{i\in I}$ be an arbitrary family of countably presented frames and
put $L=\prod_{i\in I}L_i$.  Then $\Sigma L$ is sober.
\end{theorem}

\begin{proof}
The lattice $L\times L$ is naturally the product of two copies of each $L_i$.
Applying Theorem~\ref{thm:main} to the family indexed by
$I\times\{0,1\}$, and also to the family indexed by $I$, gives
\[
  \Sigma(L\times L)
  =\left(\prod_{i\in I}\Sigma L_i\right)
   \times\left(\prod_{i\in I}\Sigma L_i\right)
  =\Sigma L\times\Sigma L.
\]
Proposition~\ref{prop:joint-sup-sober} now gives the sobriety of $\Sigma L$.
\end{proof}

\begin{corollary}\label{cor:countable-sober}
Every countable frame has a sober Scott space.
\end{corollary}

\begin{proof}
Apply Proposition~\ref{prop:countable-cp} and
Theorem~\ref{thm:product-sober} to a one-element index set.
\end{proof}

\begin{remark}\label{rem:debrecht-kawai}
For a single countably presented frame, the sobriety conclusion is already
contained in \citet[Corollary~8.6]{deBrechtKawai2019}.
Theorem~\ref{thm:product-sober} is stronger in a different direction: by
Theorem~\ref{thm:product-cp}, an arbitrary product with uncountably many
nontrivial factors is not countably presented, yet its Scott space remains
sober.
\end{remark}

\section{Countably generated spatial frames and cardinal spectra}
\label{sec:cardinal}

Theorem~\ref{thm:main} and Corollary~\ref{cor:countable-sober} give positive
results under countable presentation and countability, respectively. We
now examine the weaker condition of countable generation and the
cardinalities of spatial frames for which Scott sobriety fails. Write
$\cc=2^{\aleph_0}$ for the cardinality of the continuum.

\subsection{Countable generation and countably based representations}

The following elementary observation identifies countable generation
with the existence of a countably based spatial representation.

\begin{lemma}\label{lem:generated-based}
Let $L$ be a spatial frame. The following are equivalent:
\begin{enumerate}
\item $L$ is countably generated;
\item $L\cong\OO(X)$ for some countably based space $X$;
\item every space $X$ with $L\cong\OO(X)$ is countably based.
\end{enumerate}
\end{lemma}

\begin{proof}
Suppose that a countable set $G$ generates $L$, and fix an isomorphism
$L\cong\OO(X)$. The images of the finite meets of members of $G$, including
the empty meet, form a countable family $\mathcal B$ closed under finite
intersections. Arbitrary unions of members of $\mathcal B$ form a subframe
of $\OO(X)$ containing the images of $G$. Hence every open set is such a
union, so $\mathcal B$ is a base for $X$. This proves (1) $\Rightarrow$ (3).
Spatiality gives (3) $\Rightarrow$ (2). Finally, a countable base of $X$
generates $\OO(X)$ under arbitrary joins, proving (2) $\Rightarrow$ (1).
\end{proof}

We use the following cardinality dichotomy.

\begin{theorem}[Shelah {\citep[Theorem~1]{Shelah1993}}]
\label{thm:shelah}
If $X$ is a countably based space, then $\OO(X)$ is either countable or
has cardinality $\cc$.
\end{theorem}

\begin{proposition}\label{prop:countably-based-barrier}
Let $X$ be a countably based space. If $\Sigma\OO(X)$ is non-sober, then
\[
  |\OO(X)|=\cc.
\]
Equivalently, every countably generated spatial frame with a non-sober
Scott space has cardinality $\cc$.
\end{proposition}

\begin{proof}
If $\OO(X)$ were countable, Corollary~\ref{cor:countable-sober} would make
its Scott space sober. Thus $\OO(X)$ is uncountable, and
Theorem~\ref{thm:shelah} gives its cardinality. The equivalent formulation
follows from Lemma~\ref{lem:generated-based}.
\end{proof}

\begin{remark}\label{rem:barrier}
Assume $\cc>\aleph_1$. A spatial frame of cardinality $\aleph_1$ with a
non-sober Scott space cannot be countably generated. In particular, no
spatial representation of it can have a countable base. Thus a
countably based construction cannot settle the $\aleph_1$ problem in a
model with $\cc>\aleph_1$.
\end{remark}

For a poset $P$, its \emph{lower topology} $\omega(P)$ is generated by the
subbasic sets $P\setminus\mathord{{\uparrow}}x$, $x\in P$, where
$\mathord{{\uparrow}}x=\{y\in P:x\leq y\}$. He and Zhao obtained the
following example.

\begin{theorem}[He--Zhao {\citep[Proposition~4.11]{HeZhao2026}}]
\label{thm:hezhao}
There exists a countable dcpo $\widehat P$ such that the frame
$\omega(\widehat P)$ has a non-sober Scott space.
\end{theorem}

Here $\omega(\widehat P)$ is ordered by inclusion and is regarded as the
open-set frame of the space $(\widehat P,\omega(\widehat P))$. The Scott
topology in the statement is imposed on this frame, not on the underlying
dcpo $\widehat P$. Theorem~\ref{thm:hezhao} provides the following
contrast with the product theorem.

\begin{corollary}\label{cor:generated-counterexample}
There exists a countably generated spatial frame $L_0$ such that
\begin{equation}\label{eq:generated-counterexample}
  |L_0|=\cc,~ \Sigma L_0\text{ is non-sober},
\end{equation}
and
\begin{equation}\label{eq:generated-strict-square}
  \sigma(L_0)\times\sigma(L_0)
  \subsetneq\sigma(L_0\times L_0).
\end{equation}
Consequently, ``countably presented'' cannot be replaced by ``countably
generated'' in Theorem~\ref{thm:main}, even for spatial frames and a
two-fold self-product.
\end{corollary}

\begin{proof}
Take $L_0=\omega(\widehat P)$ from Theorem~\ref{thm:hezhao}, and put
$X_0=(\widehat P,\omega(\widehat P))$. Since $\widehat P$ is countable,
its lower topology has a countable subbase and hence a countable base.
Thus $L_0=\OO(X_0)$ is spatial and countably generated by
Lemma~\ref{lem:generated-based}. Its Scott space is non-sober by
Theorem~\ref{thm:hezhao}, and
Proposition~\ref{prop:countably-based-barrier} gives $|L_0|=\cc$.
Equation~\eqref{eq:generated-strict-square} follows from
Corollary~\ref{cor:product-strict}. The last assertion follows by comparison
with Theorem~\ref{thm:main}.
\end{proof}

\begin{remark}\label{rem:presentation-generation}
The frame $L_0$ is not countably presented, despite having a countable
generating set. On the other hand, Corollary~\ref{cor:uncountable-power-notcp}
provides frames that are not countably presented but still satisfy the
Scott-product identity for their squares. Thus countable presentation is
a sufficient, not a necessary, condition, whereas countable generation
is insufficient even under spatiality.
\end{remark}

\subsection{Upward propagation and the least cardinalities}

\begin{definition}\label{def:spectra}
Let $\Specfr$ be the class of infinite cardinals $\kappa$ for which there
is a frame $L$ with $|L|=\kappa$ and non-sober Scott space $\Sigma L$.
Let $\Specsp$ be the corresponding class with the additional requirement
that $L$ be spatial. By Corollary~\ref{cor:generated-counterexample}, both
classes are nonempty. Put
\[
  \sfr=\min\Specfr,~ \ssp=\min\Specsp.
\]
\end{definition}

Since every spatial frame is a frame, and every countable frame has a
sober Scott space, we have
\[
  \aleph_1\leq\sfr\leq\ssp.
\]
The following retraction argument shows that these two least cardinals
determine the entire spectra.

\begin{proposition}[Upward propagation]\label{prop:upward}
Let $L$ be a frame with non-sober Scott space, and let $M$ be any frame.
Then $\Sigma(L\times M)$ is non-sober. If $L$ and $M$ are spatial, then
$L\times M$ is spatial.
\end{proposition}

\begin{proof}
Let $0_M$ be the bottom element of $M$. Define
\[
  i:L\longrightarrow L\times M,~ i(x)=(x,0_M),
  ~
  p:L\times M\longrightarrow L,~ p(x,y)=x.
\]
Both maps preserve directed suprema and are Scott-continuous. Since
$p\circ i=\operatorname{id}_L$, the space $\Sigma L$ is a retract of
$\Sigma(L\times M)$. Sobriety is preserved by retracts, so
$\Sigma(L\times M)$ cannot be sober.

If $L\cong\OO(X)$ and $M\cong\OO(Y)$, then
$L\times M\cong\OO(X\coprod Y)$, proving spatiality.
\end{proof}

\begin{corollary}\label{cor:spectrum-upward}
Both $\Specfr$ and $\Specsp$ are upward closed. In particular,
\[
  \Specfr=\{\kappa:\kappa\text{ is a cardinal and }\kappa\geq\sfr\},
\]
and
\[
  \Specsp=\{\kappa:\kappa\text{ is a cardinal and }\kappa\geq\ssp\}.
\]
\end{corollary}

\begin{proof}
Let $L$ be a frame with non-sober Scott space and $|L|=\lambda$. Fix an
infinite cardinal $\kappa\geq\lambda$. Give the ordinal $\kappa$ the
topology
\[
  \tau_\kappa=\{[0,\alpha):\alpha\leq\kappa\}
\]
of initial segments, and set
$C_\kappa=\OO(\kappa,\tau_\kappa)$. This is a spatial frame of cardinality
$\kappa$, isomorphic to the complete ordinal chain $\kappa+1$.
Consequently, $|L\times C_\kappa|=\kappa$ and
$\Sigma(L\times C_\kappa)$ is non-sober by
Proposition~\ref{prop:upward}. If $L$ is spatial, so is the product.
This proves upward closure for both spectra and hence the displayed
identities.
\end{proof}

\begin{theorem}\label{thm:zfc-bounds}
The least cardinalities of Scott non-sober frames and spatial frames
satisfy
\begin{equation}\label{eq:zfc-bounds}
  \aleph_1\leq\sfr\leq\ssp\leq\cc.
\end{equation}
In particular, in ZFC there is a spatial frame with a non-sober Scott
space at every cardinality $\kappa\geq\cc$.
\end{theorem}

\begin{proof}
The lower bound follows from Corollary~\ref{cor:countable-sober}, and
$\sfr\leq\ssp$ holds by definition. The upper bound is given by
Corollary~\ref{cor:generated-counterexample}. The last assertion follows
from Corollary~\ref{cor:spectrum-upward}.
\end{proof}

For the upper bound in \eqref{eq:zfc-bounds} alone, Shelah's theorem is
not needed: $L_0=\omega(\widehat P)$ is a family of subsets of a countable
set, so $|L_0|\leq\cc$. Theorem~\ref{thm:shelah} gives the exact value
$|L_0|=\cc$ and the general restriction on countably based
representations in Proposition~\ref{prop:countably-based-barrier}.

\begin{corollary}[CH]\label{cor:ch}
Assume the Continuum Hypothesis. Then
\[
  \sfr=\ssp=\aleph_1.
\]
For every uncountable cardinal $\kappa$, there is a spatial frame
$L_\kappa$ of cardinality $\kappa$ such that $\Sigma L_\kappa$ is
non-sober and
\[
  \sigma(L_\kappa)\times\sigma(L_\kappa)
  \subsetneq\sigma(L_\kappa\times L_\kappa).
\]
\end{corollary}

\begin{proof}
Under CH, $\cc=\aleph_1$, so Theorem~\ref{thm:zfc-bounds} gives the stated
equalities. Apply Corollary~\ref{cor:spectrum-upward} for the existence
of $L_\kappa$ and Corollary~\ref{cor:product-strict} for the strict
Scott-product inequality.
\end{proof}

\section{Artinian web spaces and meet-continuous dcpos}\label{sec:artinian}

The frame condition combines finite distributivity with meet continuity.
Before asking whether the countable-frame results extend to
meet-continuous complete lattices, we establish a positive result under
an independent order-theoretic restriction. No countability assumption
is needed in this section. The closed-set formulation of meet continuity
also leads to a result for $d$-spaces whose topology is not initially
assumed to be the Scott topology.

\subsection{Meet continuity and web spaces.}
For a subset $A$ of a poset $P$, write
${\downarrow} A=\{p\in P:p\leq a\text{ for some }a\in A\}$ and define
${\uparrow} A$ dually. For a $T_0$ space $X=(X,\tau)$, let $\Omega X$
denote its specialization order, given by
$x\leq_\tau y$ if and only if $x\in\overline{\{y\}}^{\tau}$.
Write $\Gamma(X)$ for the complete lattice of closed subsets of $X$,
ordered by inclusion. Its meets are intersections and its joins are
closures of unions. Upper and lower sets in a space always refer to
its specialization order.

Kou, Liu and Luo \citep{KouLiuLuo2003} extended meet continuity from
complete lattices to arbitrary dcpos. This definition does not
presuppose binary meets. We use its poset version, as studied by Mao
and Xu \citep{MaoXu2009}: a poset $P$ is \emph{meet-continuous} if, for
every directed subset $D\subseteq P$ whose supremum exists and every
$x\leq\bigvee D$,
\begin{equation}\label{eq:mc-closure}
  x\in\operatorname{cl}_{\sigma(P)}
       ({\downarrow} D\cap{\downarrow} x).
\end{equation}
A meet-continuous dcpo is also called a \emph{meet-continuous domain}.
The following characterization relates this order-theoretic condition
to complete Heyting algebras and to Scott-open sets.

\Needspace{10\baselineskip}
\begin{theorem}[Meet-continuity characterization]\label{thm:mc-characterization}
Let $P$ be a poset. The following conditions are equivalent:
\begin{enumerate}
\item[(1)] $P$ is meet-continuous;
\item[(2)] $\Gamma(\Sigma P)$ is a complete Heyting algebra;
\item[(3)] for every $x\in P$ and $U\in\sigma(P)$,
${\uparrow}({\downarrow} x\cap U)\in\sigma(P)$.
\end{enumerate}
If $P$ is a directed-complete meet-semilattice, these conditions are
also equivalent to
\begin{equation}\label{eq:mc-semilattice}
  x\wedge\bigvee D=\bigvee_{d\in D}(x\wedge d)
\end{equation}
for every $x\in P$ and every directed subset $D\subseteq P$.
\end{theorem}

\begin{proof}[Reference]
The equivalence of (1) and (2) for dcpos is due to Kou, Liu and Luo
\citep{KouLiuLuo2003}; for arbitrary posets see
\citet[Theorem~3.8]{MaoXu2009}. The open-set formulation (3) is also
obtained from Ern\'e's characterization of web spaces
\citep[Theorem~3]{Erne2009}; see the formulation below.
For the last assertion, put $E=\{x\wedge d:d\in D\}$. Then $E$ is
directed and ${\downarrow} E={\downarrow} x\cap{\downarrow} D$, so its Scott
closure is ${\downarrow}\bigvee E$. If $P$ is meet-continuous, apply
\eqref{eq:mc-closure} to $y=x\wedge\bigvee D$. Since $y\leq x$, this
gives $y\leq\bigvee E$; the reverse inequality is immediate. Thus
\eqref{eq:mc-semilattice} holds. Conversely, when $x\leq\bigvee D$,
this identity gives $\bigvee E=x$, proving \eqref{eq:mc-closure}.
\end{proof}

More generally, Ern\'e's characterization
\citep[Theorem~3]{Erne2009} says that a $T_0$ space $(X,\tau)$ is a
\emph{web space} if and only if
\begin{equation}\label{eq:web-open}
  U\in\tau,~ x\in U
  ~\Longrightarrow~
  {\uparrow}(U\cap{\downarrow} x)\in\tau.
\end{equation}
Equivalently, $\Gamma(X)$ is a complete Heyting algebra. In the Scott
case the condition for $x\notin U$ is automatic, since then
${\downarrow} x\cap U=\varnothing$. Thus
Theorem~\ref{thm:mc-characterization} can also be read as saying that
$P$ is meet-continuous precisely when $\Sigma P$ is a web space.
For complete lattices, \eqref{eq:mc-semilattice} recovers the directed
distributive law in Section~\ref{sec:prelim}.

\subsection{Artinian spaces and algebraic domains.}
A poset is \emph{Artinian}, also called \emph{dual Noetherian}, if every
nonempty subset has a minimal element; in ZFC this is equivalent to
the descending chain condition. Notice that a minimal element need
not be a least element.

An \emph{open core} is an open principal upset. Following Ern\'e
\citep[Section~5]{Erne2009}, a space with a base of open cores is called
a \emph{$B$-space}, or a \emph{supercompactly based space}. Put
\[
  K_\tau(X)=\{k\in X:{\uparrow} k\in\tau\}.
\]
These are the $\tau$-compact points; for Scott spaces of dcpos they
are precisely the compact elements of the order.

\begin{theorem}\label{thm:artinian-web}
Every Artinian $T_0$ web space is a $B$-space. More explicitly, if
$(X,\tau)$ is a $T_0$ web space whose specialization order is Artinian,
then for every $x\in U\in\tau$ there is $k\in K_\tau(X)$ such that
\begin{equation}\label{eq:artinian-core}
  k\leq_\tau x,~ x\in{\uparrow} k\subseteq U.
\end{equation}
\end{theorem}

\begin{proof}
Fix $x\in U\in\tau$. The set $U\cap{\downarrow} x$ is nonempty, so
Artinianity provides a minimal element $k$ of this set. Then
$U\cap{\downarrow} k=\{k\}$: if $z\in U$ and $z\leq_\tau k$, then
$z\in U\cap{\downarrow} x$ and minimality forces $z=k$.
Applying \eqref{eq:web-open} at $k$ gives
\[
  {\uparrow} k={\uparrow}(U\cap{\downarrow} k)\in\tau.
\]
Since $k\leq_\tau x$ and $U$ is an upper set containing $k$, we have
$x\in{\uparrow} k\subseteq U$. Thus open cores form a base.
\end{proof}

Recall that $k$ in a dcpo $P$ is \emph{compact} if, for every directed
$D$ with $k\leq\bigvee D$, there exists $d\in D$ such that $k\leq d$.
Equivalently, ${\uparrow} k$ is Scott open. Write $K(P)$ for the compact
elements of $P$. A dcpo is an \emph{algebraic domain} if
$K(P)\cap{\downarrow} x$ is directed with supremum $x$ for every $x\in P$.
A $T_0$ space $X=(X,\tau)$ is a \emph{$d$-space}, or a
\emph{monotone convergence space}, if $\Omega X$ is a dcpo and
$\tau\subseteq\sigma(\Omega X)$; see \citet{GierzEtAl2003}.
In particular, every Scott space of a dcpo is a $d$-space.

\begin{theorem}\label{thm:artinian-dspace}
Let $X=(X,\tau)$ be a $d$-space. Suppose that $\Gamma(X)$ is a
complete Heyting algebra and that $\Omega X$ is Artinian. Then
$\Omega X$ is an algebraic dcpo and
\begin{equation}\label{eq:dspace-scott}
  \tau=\sigma(\Omega X).
\end{equation}
Consequently, $X$ is sober.
\end{theorem}

\begin{proof}
Put $P=\Omega X$. Since $X$ is a $d$-space, $P$ is a dcpo and
$\tau\subseteq\sigma(P)$. The assumption on $\Gamma(X)$ makes $X$
a web space. Theorem~\ref{thm:artinian-web} therefore supplies a base
of open cores.

For $x\in X$, put $B_x=K_\tau(X)\cap{\downarrow} x$. Applying the
basis property to the neighborhood $X$ shows that $B_x$ is nonempty.
If $k_1,k_2\in B_x$, apply it to the open neighborhood
${\uparrow} k_1\cap{\uparrow} k_2$ of $x$. There is $k_3\in B_x$ with
\[
  x\in{\uparrow} k_3\subseteq{\uparrow} k_1\cap{\uparrow} k_2,
\]
so $k_1,k_2\leq_\tau k_3$. Thus $B_x$ is directed.

Let $s=\bigvee B_x\leq_\tau x$. If $s\ne x$, the set
$X\setminus{\downarrow} s$ is a $\tau$-open neighborhood of $x$,
because ${\downarrow} s=\overline{\{s\}}^{\tau}$. Its basis refinement
gives $k\in B_x$ with
${\uparrow} k\subseteq X\setminus{\downarrow} s$, contrary to $k\leq_\tau s$.
Hence
\[
  x=\bigvee B_x.
\]

Let $V\in\sigma(P)$ and $x\in V$. Since $B_x$ is directed with
supremum $x$, some $k\in B_x$ lies in $V$. Then
$x\in{\uparrow} k\subseteq V$ and ${\uparrow} k\in\tau$. Thus every
Scott-open set is $\tau$-open, proving \eqref{eq:dspace-scott}.
It follows that $K_\tau(X)=K(P)$, so the preceding directed-supremum
formula says exactly that $P$ is algebraic. Finally, Scott spaces of
algebraic dcpos are sober \citep{GierzEtAl2003}, and hence so is $X$.
\end{proof}

\begin{corollary}\label{cor:artinian-mc}
Every Artinian meet-continuous dcpo $P$ is an algebraic domain.
Consequently, $\Sigma P$ is sober and
\begin{equation}\label{eq:artinian-square}
  \sigma(P\times P)=\sigma(P)\times\sigma(P).
\end{equation}
In particular, every Artinian meet-continuous complete lattice is
algebraic.
\end{corollary}

\begin{proof}
By Theorem~\ref{thm:mc-characterization}, $\Gamma(\Sigma P)$ is a
complete Heyting algebra. Since $\Sigma P$ is a $d$-space,
Theorem~\ref{thm:artinian-dspace} gives algebraicity and sobriety.

For the product identity, put $K_x=K(P)\cap{\downarrow} x$ for each
$x\in P$. Let $W\in\sigma(P\times P)$ contain $(x,y)$.
The directed set $K_x\times K_y$ has supremum $(x,y)$, so some
$(k,l)\in K_x\times K_y$ belongs to $W$. Then
${\uparrow} k\times{\uparrow} l$ is a product-open neighborhood of $(x,y)$
contained in $W$. This proves the nontrivial inclusion in
\eqref{eq:artinian-square}; the reverse inclusion is
\eqref{eq:general-inclusion}.
\end{proof}

Taking the contrapositive of the sobriety assertion gives the following
restriction on every dcpo with a non-sober Scott space.

\begin{theorem}[Dichotomy theorem for Scott non-sober dcpos]\label{thm:dichotomy}
Let $P$ be a dcpo whose Scott space $\Sigma P$ is non-sober.
Then at least one of the following holds:
\begin{enumerate}
\item[(1)] $P$ is not meet-continuous;
\item[(2)] $P$ is not Artinian.
\end{enumerate}
\end{theorem}

\begin{proof}
If $P$ were both meet-continuous and Artinian, then
Corollary~\ref{cor:artinian-mc} would make $\Sigma P$ sober.
\end{proof}

The two alternatives are not mutually exclusive. In particular,
non-sobriety alone cannot be used to conclude failure of meet continuity:
Artinianity is needed for that deduction from this theorem. Conversely,
for a meet-continuous dcpo, non-sobriety forces an infinite strictly
descending chain.

\subsection{Seven non-sober examples in the light of the dichotomy.}
We distinguish consequences of Theorem~\ref{thm:dichotomy} from
properties of the individual constructions.

\emph{The Artinian examples.}
Johnstone's dcpo \citep{Johnstone1981} and Jia's countable dcpo
\citep[Example~2.6.1]{Jia2018} are Artinian. In either construction, a
strictly descending chain has, after at most one step, fixed column
labels and a strictly decreasing natural-number height. Since their
Scott spaces are non-sober, Theorem~\ref{thm:dichotomy} shows that
neither dcpo is meet-continuous.

For the extended Johnstone dcpo of Zhao, Xi and Chen
\citep{ZhaoXiChen2019}, the same observation takes an ordinal form.
In its coordinate description,
$\mathcal Z=\omega_1\times(\omega_1+1)$ and
\begin{equation}\label{eq:extended-johnstone-order}
\begin{split}
  (\alpha,\beta)\leq(\gamma,\delta)
  ~\Longleftrightarrow~
  &(\alpha=\gamma\text{ and }\beta\leq\delta)\\
  &\text{or }(\delta=\omega_1\text{ and }\beta\leq\gamma).
\end{split}
\end{equation}
The map $r(\alpha,\beta)=\beta$ is strictly increasing on strict
comparisons: $x<y$ implies $r(x)<r(y)$. Indeed, a strict comparison in
one column increases the second coordinate, while a comparison between
different columns ends at height $\omega_1$ and starts below it.
For any nonempty $A\subseteq\mathcal Z$, choose $a\in A$ whose rank is
least in $r[A]$. Then $a$ is minimal in $A$. Thus $\mathcal Z$ is
Artinian, using only the well-ordering of the ordinals, not the
regularity of $\omega_1$. Its non-sober Scott space and
Theorem~\ref{thm:dichotomy} imply that $\mathcal Z$ is not meet-continuous.

The countable distributive complete lattice of Miao, Xi, Li and Zhao
\citep{MiaoXiLiZhao2023} is also Artinian and has a non-sober Scott
space, so the same theorem applies. For this lattice there is an
independent deduction within the present paper: meet continuity together
with finite distributivity would make it a countable frame, contrary to
Corollary~\ref{cor:countable-sober}.

\emph{Isbell's complete lattice.}
We next include Isbell's lattice \citep{Isbell1982} in the Artinian
class. We use Goubault-Larrecq's detailed account
\citep{GoubaultLarrecqIsbell}, where well-foundedness of the finite
intersections of principal ideals is already noted. The following
argument makes the passage to arbitrary intersections explicit. That
account uses positive real numbers in place of Isbell's ordinal indices;
the proof below depends only on the block structure and the injectivity
of the attaching map, and applies to either indexing.

\begin{proposition}\label{prop:isbell-artinian}
Isbell's complete lattice is Artinian and is not meet-continuous.
\end{proposition}

\begin{proof}
Write $P$ for the dcpo before completion. In each block $G_\alpha$,
there are columns $(m_k^\alpha)_{m\geq1}$ indexed by positive integers
$k$, with respective suprema $\omega_k^\alpha$, and columns
$(n_f^\alpha)_{n\geq1}$ indexed by maps $f:\N_{>0}\to\N_{>0}$.
The elements $\omega_1^\alpha<\omega_2^\alpha<\cdots$ have supremum
$\Omega_\alpha$, which is also the supremum of each $f$-column.
The two families of columns have no comparisons with each other below
$\Omega_\alpha$. Relations between distinct blocks only put finite
initial segments of columns below a maximal element $\Omega_\gamma$.
By injectivity of the attaching map, at most two such segments, from
distinct other blocks, occur below any given $\Omega_\gamma$.
These are the structural properties needed from the construction
\citep{Isbell1982,GoubaultLarrecqIsbell}.

Put $W_{\alpha,k}={\downarrow}\omega_k^\alpha$ and
$C_\alpha={\downarrow}\Omega_\alpha$. Principal ideals generated by
column elements $m_k^\alpha$ or $n_f^\alpha$ are finite.
Each $W_{\alpha,k}$ lies within $G_\alpha$, and
$C_\alpha\setminus G_\alpha$ is finite. Hence the intersection of
distinct $C_\alpha$'s is finite; the intersection of a $C_\alpha$
with a $W_{\beta,k}$ is either $W_{\beta,k}$ or finite; and two
$W$-sets either belong to the same block and are nested, or are disjoint.
Consequently every finite intersection of principal ideals, including
the empty intersection, is finite, a $W_{\alpha,k}$, a $C_\alpha$, or $P$.
Let $\mathcal F$ denote the family of these finite intersections.
The ordinal-valued map
\begin{equation}\label{eq:isbell-rank}
  \rho(A)=
  \begin{cases}
    |A|,&\text{if }A\text{ is finite},\\
    \omega+k,&\text{if }A=W_{\alpha,k},\\
    \omega\cdot2,&\text{if }A=C_\alpha,\\
    \omega\cdot2+1,&\text{if }A=P
  \end{cases}
\end{equation}
is strictly increasing under proper inclusion. Indeed, distinct
$C_\alpha$'s are incomparable, and proper inclusion between $W$-sets
can occur only within one block, with increasing $k$. Thus
$(\mathcal F,\subseteq)$ is Artinian.

The completed lattice, ordered by inclusion, is
\[
  L=\left\{\bigcap_{p\in S}{\downarrow}p:S\subseteq P\right\}.
\]
Fix $S\subseteq P$ and consider
\[
  \mathcal F_S=\left\{\bigcap_{p\in E}{\downarrow}p:
               E\subseteq S\text{ finite}\right\}.
\]
This family is nonempty and closed under binary intersections, and so
has a minimal member $A_0$. For every $B\in\mathcal F_S$,
minimality gives $A_0\cap B=A_0$. Hence $A_0$ is the least member of
$\mathcal F_S$, and
\[
  \bigcap_{p\in S}{\downarrow}p=\bigcap\mathcal F_S=A_0.
\]
Thus $L=\mathcal F$ is Artinian. Since $\Sigma L$ is non-sober,
Theorem~\ref{thm:dichotomy} shows that $L$ is not meet-continuous.

There is also a direct witness to the failure of meet continuity.
Fix one block $\alpha$ and one map $f$, and set
$a={\downarrow}1_1^\alpha$ and $b_n={\downarrow}n_f^\alpha$
for $n\geq1$. The set $\{b_n:n\geq1\}$ is directed in $L$,
$a\cap b_n=\varnothing$, and $\bigvee_{n\geq1}b_n=C_\alpha$.
For the last equality, any principal ideal containing every
$n_f^\alpha$ must contain their supremum $\Omega_\alpha$ in $P$;
the same holds for any intersection of such principal ideals.
As meets in $L$ are intersections and $0_L=\varnothing$, we obtain
\begin{equation}\label{eq:isbell-mc-failure}
  a\wedge\bigvee_{n\geq1}b_n
  =a>0_L=\bigvee_{n\geq1}(a\wedge b_n).
\end{equation}
\end{proof}

\emph{The Xu--Xi--Zhao frame.}
The complete Heyting algebra constructed in \citet{XuXiZhao2021} is
meet-continuous and has a non-sober Scott space. Hence
Theorem~\ref{thm:dichotomy} forces it to be non-Artinian.
It is also uncountable by Corollary~\ref{cor:countable-sober}.
Thus this example does not answer the countable meet-continuous
sobriety question.

Kou's example requires a separate check and also shows why the
alternatives in Theorem~\ref{thm:dichotomy} may overlap.

\begin{example}[Kou's dcpo]\label{ex:kou}
The dcpo constructed by Kou \citep{Kou2001} has a well-filtered,
non-sober Scott space. We use its explicit description in
\citet[Example~3.11]{ZhaoXu2018}. Let
\[
  X=(0,1],~
  P_0=\{(k,a,b)\in\mathbb R^3:0<k<1,\ 0<b\leq a\leq1\},
\]
and let $P=X\sqcup P_0$. Distinct elements of $X$ are incomparable,
and the remaining comparisons are
\begin{align*}
  (k,a,b)\leq(k',a',b')
    &\ \Longleftrightarrow\ k\leq k',\ a=a',\ b=b',\\
  (k,a,b)\leq x\in X
    &\ \Longleftrightarrow\ a=x\ \text{or}\ kb\leq x<b.
\end{align*}
Consider
\[
  D=\{(t,1,1):0<t<1\},~ u=(1/2,1,1/2).
\]
The set $D$ is a chain with supremum $1\in X$. Indeed, it has no upper
bound in $P_0$, and the only point of $X$ above every member of $D$ is
$1$. Although $u\leq1$, the two lower sets are disjoint:
\[
  {\downarrow} D\cap{\downarrow} u=\varnothing.
\]
This follows because comparisons inside $P_0$ preserve the two labels
$(a,b)$, and no point of $X$ lies below an element of $P_0$.
Condition~\eqref{eq:mc-closure} fails, so $P$ is not meet-continuous.
Moreover,
\[
  (1/2,1,1)>(1/3,1,1)>(1/4,1,1)>\cdots
\]
is a strictly descending chain. Thus Kou's dcpo is also non-Artinian.
\end{example}

Thus the Johnstone, Jia, Zhao--Xi--Chen, Miao--Xi--Li--Zhao and Isbell
examples are Artinian and therefore not meet-continuous. The
Xu--Xi--Zhao frame is meet-continuous but non-Artinian, while Kou's
dcpo fails both conditions. In particular, the other six examples all
fail meet continuity, so none supplies a countable meet-continuous
counterexample to the sobriety questions below.

\section{Further questions}\label{sec:questions}

There are now two positive results to compare: countable frames have
sober Scott spaces and satisfy the Scott-product identity, while
Artinian meet-continuous dcpos have both properties by
Corollary~\ref{cor:artinian-mc}. The first result uses finite
distributivity in addition to meet continuity; the second uses the
descending chain condition but imposes no cardinality restriction.
It is therefore natural to ask whether countability and meet continuity
alone suffice for the two conclusions in the complete-lattice case.

\begin{question}\label{ques:countable-mc-product}
Let $L$ be a countable meet-continuous complete lattice. Must
\[
  \sigma(L\times L)=\sigma(L)\times\sigma(L)?
\]
\end{question}

\begin{question}\label{ques:countable-mc-sober}
Let $L$ be a countable meet-continuous complete lattice. Must its Scott
space $\Sigma L$ be sober?
\end{question}

By Proposition~\ref{prop:joint-sup-sober}, an affirmative answer to
Question~\ref{ques:countable-mc-product} would imply an affirmative answer
to Question~\ref{ques:countable-mc-sober}. Both answers are affirmative
for countable frames by Section~\ref{sec:products}, and for Artinian
meet-continuous complete lattices by Corollary~\ref{cor:artinian-mc}.

The algebraicity criterion also motivates a more general sobriety
question, without assuming a lattice structure.

\begin{question}\label{ques:countable-mc-dcpo}
Let $P$ be a countable meet-continuous dcpo. Must its Scott space
$\Sigma P$ be sober?
\end{question}

An affirmative answer to Question~\ref{ques:countable-mc-dcpo} would
answer Question~\ref{ques:countable-mc-sober} affirmatively. The
product-to-sobriety implication from Proposition~\ref{prop:joint-sup-sober}
is used above only for complete lattices, not for arbitrary dcpos.
Any counterexample to either sobriety question must be non-Artinian by
Theorem~\ref{thm:dichotomy}. None of the seven constructions discussed
in Section~\ref{sec:artinian} supplies such a countable meet-continuous
counterexample.

One cannot strengthen these questions by asking for algebraicity:
Jia and Xi's non-continuous countable frames \citep{JiaXi2026} are
countable and meet-continuous but not algebraic. In particular, they
are non-Artinian by Corollary~\ref{cor:artinian-mc}, although their
Scott spaces are sober by Corollary~\ref{cor:countable-sober}.
Thus countability cannot replace Artinianity in the algebraicity
conclusion; the issues here are precisely sobriety and Scott-product
compatibility.

A separate issue concerns the cardinal threshold for frames themselves.
The inequalities in Theorem~\ref{thm:zfc-bounds} leave the following
questions when no hypothesis on the continuum is imposed.

\begin{question}\label{ques:cardinality}
Does ZFC prove that $\sfr=\aleph_1$? Does it prove the stronger assertion
$\ssp=\aleph_1$?
\end{question}

The spatial assertion in Question~\ref{ques:cardinality} is equivalent,
by Corollary~\ref{cor:spectrum-upward}, to the assertion that every
uncountable cardinal is the cardinality of a spatial frame with a
non-sober Scott space. Corollary~\ref{cor:ch} proves both assertions under
CH. This motivates the corresponding consistency question.

\begin{question}\label{ques:independence}
Is it consistent with ZFC that $\ssp>\aleph_1$?
\end{question}

In any such model one must have $\cc>\aleph_1$. Conversely, when
$\cc>\aleph_1$, Proposition~\ref{prop:countably-based-barrier} excludes
countably based spatial representations for a counterexample of
cardinality $\aleph_1$, but does not exclude a spatial counterexample
with an uncountable base. The results here therefore give the bounds and
the CH case, without asserting an independence theorem or an implication
from any forcing axiom.

\section*{Statements and Declarations}

\noindent\textbf{Funding.}
This research was supported by the National Natural Science Foundation of
China (Nos. 12471070, 12071199).

\vspace{0.1cm}

\noindent\textbf{Competing interests.}
The author declares no competing interests.

\vspace{0.1cm}

\noindent\textbf{Data availability.}
No datasets were generated or analyzed in this study.

\vspace{0.1cm}


\raggedbottom

\end{document}